\documentclass[12pt,a4paper]{amsart}
\usepackage[top=1.15in, bottom=1.15in, left=1.17in, right=1.17in]{geometry}
\usepackage{amssymb}
\usepackage{pdflscape}
\usepackage{amscd}
\usepackage{fancyhdr}

\newcommand{\bi}{\mathbf{i}}

\newcommand{\br}{\mathbf{r}}
\newcommand{\rk}{\mathbf{rk}}
\newcommand{\bs}{\mathbf{s}}
\newcommand{\bd}{\mathbf{d}}
\newcommand{\bp}{\mathbf{p}}
\newcommand{\bt}{\mathbf{t}}
\newcommand{\ba}{\mathbf{a}}
\newcommand{\be}{\mathbf{e}}
\newcommand{\bn}{\mathbf{n}}
\newcommand{\bb}{\mathbf{b}}
\newcommand{\bm}{\mathbf{m}}
\newcommand{\bx}{\mathbf{x}}
\newcommand{\by}{\mathbf{y}}
\newcommand{\bz}{\mathbf{z}}

\newcommand{\g}{\mathfrak{g}}
\newcommand{\ra}{\rightarrow}

\newcommand{\vp}{\varphi}
\newcommand{\bino}[2]{\left[\genfrac{}{}{0pt}{}{#1}{#2}\right]}

\newtheorem{theorem}{Theorem}[section]
\newtheorem{corollary}[theorem]{Corollary}
\newtheorem{lemma}[theorem]{Lemma}
\newtheorem{proposition}[theorem]{Proposition}

\newtheorem{definition}[theorem]{Definition}

\title{Supports for linear degenerations of flag varieties}
\author{Xin Fang, Markus Reineke}

\begin{document}
\begin{abstract}We determine the set of supports for the flat family of linear degenerations of flag varieties in terms of Motzkin combinatorics.\end{abstract}
\maketitle

\section{Introduction}

In recent years, the concept of supports of a projective map of complex algebraic varieties has received attention, originating from the Support Theorem of B. C. Ng\^o \cite{Ngo}. In non-technical terms, the supports of a projective map are a collection of locally closed subvarieties of the target space, whose singularities (more precisely, their local intersection cohomology) control the variation of cohomology of the fibres of the map (see Section \ref{subsec:supports} for a brief introduction, and \cite{Mig} for a detailed overview over the subject). It thus seems desirable to compute the supports for interesting classes of projective maps, for example those arising in contexts of algebraic Lie theory.

In the present paper, we provide an explicit description of the supports of a flat family called linear degenerations of flag varieties. Based on earlier work on degenerate versions of flag varieties \cite{CFR,Fei}, this family is introduced and studied in \cite{CFFFR}. Despite the rather simple idea to degenerate flag varieties by relaxation of the containment relation between the subspaces constituting the flag, it provides a wealth of new degenerations, which nevertheless share favourable geometric properties (see Section \ref{subsec:mainresult} for the precise definitions and results).

Our main result, Theorem \ref{thm:main} below, describes the supports of the family of linear degenerations explicitly in terms of Motzkin paths. The appearance of Motzkin combinatorics is quite surprising and not to be expected a priori. It shows that the set of supports is highly nontrivial in our case, but still completely controllable. This description also allows us to prove that the set of supports is ``asymptotically very small'' compared to the set of all potential supports, see Section \ref{subsec:asymp}.

The determination of the supports of this flat family is made possible by the observation (which in fact formed the starting point for the present work), that it features as a special case of the varieties and maps arising in G. Lusztig's geometric realization of quantized enveloping algebras \cite{Lus3} (all quantum groups notions will be recalled in Section \ref{sec:que}). This allows us to reduce the determination of supports to an algebraic problem, namely expanding a certain monomial in Chevalley generators of a quantized enveloping algebra into elements of the canonical \cite{Lus1}, or global crystal \cite{Kas}, basis (see Section \ref{sec:geometric}).

Although the inherent piecewise linear combinatorics of Lusztig's canonical basis is by now well studied \cite{BFZ, BZ, BZ2, CMM, L}, and ultimately led to the new research areas of crystal basis theory and cluster algebras/combinatorics, the basis itself remains rather mysterious, and our algebraic reduction of the support problem is still not readily solvable.

Instead, we use a parametrization of the canonical basis elements which is different from the one provided by the geometric picture, for which partial results on the desired expansion can be derived purely algebraically (see Section \ref{sec:expansion}). We then precisely use the knowledge on the piecewise linear combinatorics of the canonical basis, namely the ingenious multi-segment duality formula of Knight and Zelevinsky \cite{KZ, Z2} (to be reviewed in Section \ref{sec:msd}), to play off two dual pictures of the canonical basis against each other, which then leads to the exact determination of the supports, and to the natural appearance of Motzkin combinatorics (see Section \ref{subsec:proof}).

The limitations of this indirect approach, substantiated by further explicit examples, are discussed in Sections \ref{subsec:remarks} and Section \ref{sec:smallrank}.

\vskip 5pt
\noindent
\textbf{Acknowledgments}. The work of the authors is supported by the DFG projects TRR 191 ``Symplectic Structures in Geometry, Algebra and Dynamics'' and DFG-RSF ``Geometry and representation theory at the interface between Lie algebras and quivers''.

\section{Statement of the main result}

\subsection{Supports}\label{subsec:supports}

We give a brief introduction into the concept of supports of a projective morphism and recommend \cite{Mig} for a thorough review of the topic.

Let $f:X\rightarrow Y$ be a projective morphism of complex algebraic varieties, with $X$ assumed to be irreducible and smooth. We view this morphism as the family of its fibres $X_y:=f^{-1}(y)$. We are interested in the behaviour of the cohomology $\mathrm{H}^*(X_y;\mathbb{Q})$ of the fibre when $y$ varies along $Y$. This information is clearly encoded in the complex of constructible sheaves $\mathrm{R}f_*\mathbb{Q}_X\in\mathcal{D}^b(Y)$, since, for example, the cohomology of a fibre identifies with a stalk of its cohomology sheaves, $$\mathrm{H}^*(X_y;\mathbb{Q})\simeq\mathcal{H}_y^*(\mathrm{R}f_*(\mathbb{Q}_X)).$$
By the Decomposition Theorem \cite{BBD}, $\mathrm{R}f_*(\mathbb{Q}_X)$ is isomorphic to a direct sum of shifts of intersection cohomology complexes,
$$\mathrm{R}f_*\mathbb{Q}_X\simeq \bigoplus_{i=1}^n {\rm IC}(\overline{S_i},\mathcal{L}_i)[s_i]$$
for finitely many data $(S_i,\mathcal{L}_i,s_i)$ consisting of a smooth locally closed subvariety $S_i\subset Y$, a non-zero local system $\mathcal{L}_i$ on $S_i$, and an integer $s_i$.

\begin{definition} The set $\{\overline{S_i}\,|\, i=1,\ldots,n\}$ is called the set of supports of $f$.
\end{definition}

We then find a decomposition (up to shifts):
$$\mathrm{H}^*(X_y;\mathbb{Q})\simeq\bigoplus_{i=1}^n\mathcal{H}^*_y({\rm IC}(\overline{S_i},\mathcal{L}_i)),$$
thus the cohomology of $X_y$ is essentially controlled by the local intersection cohomology of the supports $\overline{S_i}$.

A point of view advocated in \cite{Ngo} is that the set of supports should be viewed as a topological invariant of the map $f$ in its own right. We refer to \cite{Mig} for a collection of general results on the set of supports, including a codimension estimate for supports due to Goresky and MacPherson, the description of supports for semismall maps, Ng\^o's support theorem for certain abelian filtrations, and results of Migliorini and Shende on supports and higher discriminant loci.

\subsection{Statement of the main result}\label{subsec:mainresult}

Fix $n\geq 1$ and denote by $V$ an $(n+1)$-dimensional complex vector space.

We define a family $\pi:\mathcal{F}\rightarrow R$ of so-called linear degenerations of the ${\rm GL}_{n+1}(\mathbb{C})$-flag variety. The base space for the family of degenerations is $R:={\rm Hom}_{\mathbb{C}}(V,V)^{n-1}$, on which the group $G={\rm GL}(V)^n$ acts via base change with finitely many orbits $\mathcal{O}(\mathbf{r})$, indexed by rank tuples $\mathbf{r}=(r_{ij})_{1\leq i\leq j\leq n}$ (see \cite{Z1}). Namely, to a point $f_*=(f_1,\ldots,f_{n-1})\in R$ we associate the rank tuple $\mathbf{r}(f)=({\rm rank}(f_{j-1}\circ\ldots\circ f_i))_{i\leq j}$. All these orbits have connected stabilizers.

Let $\mathrm{Gr}_i(V)$ be the Grassmann variety of $i$-dimensional subspaces in $V$. We define $\mathrm{Gr}(V)=\prod_{i=1}^n\mathrm{Gr}_i(V)$, and define 
$$\mathcal{F}=\{(U_*,f_*)\in\mathrm{Gr}(V)\times R\mid\, f_i(U_i)\subset U_{i+1},\, i=1,\ldots,n-1\}.$$

We have a canonical projection $p:\mathcal{F}\rightarrow{\rm Gr}(V)$ turning $\mathcal{F}$ into a homogeneous vector bundle over ${\rm Gr}(V)$, thus $\mathcal{F}$ is smooth and irreducible. The projection $\pi:\mathcal{F}\rightarrow R$ is a projective map, whose fibres are denoted by $\mathrm{Fl}^{f_*}(V):=\pi^{-1}(f_*)$ and called linear degenerate flag variety. More explicitly,
$$\mathrm{Fl}^{f_*}(V)=\{(U_1,\ldots,U_n)\in{\rm Gr}(V)\, |\, f_i(U_i)\subset U_{i+1},\, i=1,\ldots,n-1\}.$$
Note that $\mathrm{Fl}^{{\rm id}}(V)\simeq{\rm GL}_{n+1}({\mathbb{C}})/B$ is the type $A$ complete flag variety; more generally, a generic fibre of $\pi$ is isomorphic to the ${\rm GL}_{n+1}(\mathbb{C})$-complete flag variety. By $G$-equivariance, the isomorphism type of $\mathrm{Fl}^{f_*}(V)$ only depends on the rank tuple of $f_*$. Denote by $\mathbf{r}^1$ the special rank tuple given by $\mathbf{r}^1:=(n+1+i-j)_{1\leq i\leq j\leq n}$.

One of the main results of \cite{CFFFR} classifies the irreducible fibres of dimension $\dim({\rm GL}_{n+1}({\mathbb{C}})/B)$:

\begin{theorem}[\cite{CFFFR}]
The following statements are equivalent:
\begin{enumerate}
\item the linear degenerate flag variety $\mathrm{Fl}^{f_*}(V)$ is irreducible of dimension $n(n+1)/2$;
\item $\mathbf{r}(f_*)\geq\mathbf{r}^1$ componentwise;
\item for any $i=1,\cdots,n-1$, we have $r_{i,i+1}\in\{n,n+1\}$.
\end{enumerate}
If this is the case, $\mathrm{Fl}^{f_*}(V)$ is normal, locally a complete intersection, prehomogeneous, and admits an affine paving.
\end{theorem}

We denote by $U\subset R$ the set of all $f_*$ such that $\mathbf{r}(f_*)\geq \mathbf{r}^1$ and by $\mathcal{F}_U$ the pre-image of $U$ under $\pi$; thus $\pi:\mathcal{F}_U\rightarrow U$ is a projective map between smooth irreducible varieties which is flat with irreducible fibres.

By the Decomposition Theorem \cite{BBD} (using $G$-equivariance and connectedness of stabilizers), $\mathrm{R}\pi_*\mathbb{Q}_{\mathcal{F}_U}$ decomposes into a direct sum of shifts of  intersection cohomology complexes (with respect to trivial local systems) ${\rm IC}(\overline{\mathcal{O}(\mathbf{r})})$.

Denote by $\mathcal{M}_n$ the set of Motzkin paths from $(0,0)$ to $(n,0)$, that is, $\mathcal{M}_n$ is the set of all tuples of nonnegative integers
$$\mathbf{x}=(0=x_0,x_1,\ldots,x_{n-1},x_n=0)$$
such that
$$x_i-x_{i-1}\in\{-1,0,1\},\, i=1,\ldots,n.$$

To such a Motzkin path $\mathbf{x}\in\mathcal{M}_n$ we associate a tank tuple $\mathbf{r}(\mathbf{x})\geq\mathbf{r}^1$, where
$$\mathbf{r}(\mathbf{x})_{ij}=n+1-\max_{i\leq k\leq l\leq m\leq j}(x_{l-1}+x_l-x_{k-1}-x_m).$$

Our main result gives a complete description of the set of supports of the map $\pi$:

\begin{theorem}\label{thm:main} The intersection complex ${\rm IC}(\overline{\mathcal{O}(\mathbf{r})})$ appears (up to shift) as a direct summand of $\mathrm{R}\pi_*\mathbb{Q}_{\mathcal{F}_U}$
if and only if $\mathbf{r}$ is of the form $\mathbf{r}(\mathbf{x})$ for a Motzkin path $\mathbf{x}\in \mathcal{M}_n$.
\end{theorem}

The proof of this theorem is given in Section \ref{subsec:proof}.

\subsection{An elementary example}\label{subsec:elex} Taking $V={\mathbb{C}}^3$ in the previous section and identifying ${\rm Gr}_1(V)\simeq\mathbb{P}(V)$, ${\rm Gr}_2(V)\simeq\mathbb{P}(V^*)$, we have
$$\mathcal{F}\simeq \{(v,\varphi,f)\in{\mathbb{P}}(V)\times{\mathbb{P}}(V^*)\times{\rm Hom}(V,V)\, :\, \varphi\circ f\circ v=0\}.$$
Using the ${\rm GL}_3(\mathbb{C})\times{\rm GL}_3(\mathbb{C})$-equivariance, we see that $\mathrm{Fl}^{(f)}(V)$ only depends on the rank $r\in\{0,1,2,3\}$ of $f$. Choosing coordinates and diagonalizing $f$, we arrive at the following four types of degenerations:
\begin{enumerate}
\item when the rank $r=3$,
$$\mathrm{Fl}^{(3)}(V)\simeq\{((x_0:x_1:x_2),(y_0:y_1:y_2))\in\mathbb{P}^2\times\mathbb{P}^2\, :\, x_0y_0+x_1y_1+x_2y_2=0\},$$
is the ${\rm GL}_3(\mathbb{C})$-flag variety, thus an irreducible three-dimensional smooth projective variety;
\item when the rank $r=2$,
$$\mathrm{Fl}^{(2)}(V)\simeq\{((x_0:x_1:x_2),(y_0:y_1:y_2))\in\mathbb{P}^2\times\mathbb{P}^2\, :\, x_0y_0+x_1y_1=0\},$$
is an irreducible normal three-dimensional projective variety with an isolated singularity;
\item when the rank $r=1$,
$$\mathrm{Fl}^{(1)}(V)\simeq\{((x_0:x_1:x_2),(y_0:y_1:y_2))\in\mathbb{P}^2\times\mathbb{P}^2\, :\, x_0y_0=0\},$$
is isomorphic to $(\mathbb{P}^2\times\mathbb{P}^1)\vee_{\mathbb{P}^1\times\mathbb{P}^1}(\mathbb{P}^2\times\mathbb{P}^1)$, thus a three-dimensional reducible non-normal projective variety with two irreducible components intersecting in codimension one;
\item when the rank $r=0$,
$$\mathrm{Fl}^{(0)}(V)\simeq\mathbb{P}^2\times\mathbb{P}^2,$$
is four-dimensional.
\end{enumerate}

Denoting by $\mathcal{O}(r)\subset{\rm End}(V)$ the locus of linear maps of rank $r$, we see that $\pi:\mathcal{F}\ra\mathrm{End}(V)$ is flat over $\mathcal{O}(3)\cup\mathcal{O}(2)\cup\mathcal{O}(1)$, and is flat with irreducible fibres over $\mathcal{O}(3)\cup\mathcal{O}(2)$.\\[1ex]
Borrowing a result from Section \ref{subsec:rank2}, we have
$$\mathrm{R}\pi_*\mathbb{Q}_\mathcal{F}\simeq{\rm IC}(\overline{\mathcal{O}(3)})\otimes \mathrm{H}^*({\rm GL}_3(\mathbb{C})/B)\oplus{\rm IC}(\overline{\mathcal{O}(2)}).$$

\section{Quantum groups and canonical bases}\label{sec:que}

Let $A$ be the Cartan matrix of the Lie algebra $\g=\mathfrak{sl}_{n+1}$ over $\mathbb{C}$, and let $\g=\mathfrak{n}^-\oplus\mathfrak{h}\oplus\mathfrak{n}^+$ be the triangular decomposition. Let $\alpha_1,\cdots,\alpha_n$ be the corresponding simple roots and 
$$\Delta_+=\{\alpha_{i,j}=\alpha_i+\cdots+\alpha_j\mid 1\leq i\leq j\leq n\}$$
be the set of positive roots of $\g$. We define $N:=\#\Delta_+={n(n+1)}/{2}$.

Let $w_0$ be the longest element in $\mathfrak{S}_{n+1}$, the Weyl group of $\g$, and let $R(w_0)$ be the set of all reduced decompositions of $w_0$.

Let $\bi=(i_1,\cdots,i_N)\in R(w_0)$ be such a reduced decomposition of $w_0$. Then the set $\{\beta_1,\beta_2,\cdots,\beta_N\}$, where $\beta_k=s_{i_1}\cdots s_{i_{k-1}}(\alpha_{i_k})$, coincides with $\Delta_+$. This induces a convex ordering on $\Delta_+$ by letting
$$\beta_1<\beta_2<\cdots<\beta_N.$$

For a variable $v$, let $U_v(\g)$ be the quantized enveloping algebra associated to $\g$ over $\mathbb{Q}(v)$, with generators $E_i$, $F_i$ and $K_i^{\pm 1}$ for $i=1,\cdots,n$. Let $U_v^+(\g)$ be the positive part of $U_v(\g)$: it is the $\mathbb{Q}(v)$-subalgebra of $U_v(\g)$ generated by the Chevalley generators $E_1,E_2,\cdots,E_n$, which is isomorphic to the free associative algebra in $E_1,\cdots,E_n$ subject to the quantum Serre relations: for $|i-j|>1$, the generators $E_i$ and $E_j$ commute, and for $|i-j|=1$, we have
$$E_i^2E_j-(v+v^{-1})E_iE_jE_i+E_jE_i^2=0.$$

We fix the following notation: for $k\in\mathbb{N}$ and $0\leq m\leq k$, we define quantum numbers, quantum factorials and quantum binomial coefficients by
$$[k]_v:=\frac{v^k-v^{-k}}{v-v^{-1}},\ \ [k]_v!=[k]_v[k-1]_v\cdots[1]_v,\ \ \bino{k}{m}_v=\frac{[k]_v!}{[m]_v![k-m]_v!},$$
as well as the divided powers of the Chevalley generators:
$$E_i^{(r)}=\frac{E_i^r}{[r]_v!}.$$
Let $U_v^+$ denote the $\mathbb{Z}[v,v^{-1}]$-subalgebra of $U_v^+(\g)$ generated by the divided powers $E_i^{(r)}$ for $i=1,\cdots,n$ and $r\in\mathbb{N}$.

Let $\overline{\cdot}:U_v^+(\g)\ra U_v^+(\g)$ be the bar involution: it is the $\mathbb{Q}$-algebra map on $U_v^+(\g)$ defined by
$$v\mapsto v^{-1},\ \ E_i\mapsto E_i\ \ \text{for }i=1,\cdots,n.$$

\subsection{PBW basis of $U_v^+(\g)$}

For any $i=1,\cdots,n$, Lusztig defined on $U_v(\g)$ an algebra automorphism $T_i:U_v(\g)\ra U_v(\g)$ (see for example \cite[Chapter 37]{Lus2}).

We fix a reduced decomposition $\bi=(i_1,\cdots,i_N)$. For $k=1,\cdots,N$ and $m\in\mathbb{N}$, define 
$$E_{\beta_k}^{(m)}:=T_{i_1}T_{i_2}\cdots T_{i_{k-1}}(E_{i_k}^{(m)})$$
which plays the role of a divided power of a PBW root vector associated to $\beta_k$. 

For $\mathbf{m}=(m_1,m_2,\cdots,m_N)\in\mathbb{N}^N$, we denote
$$E_\bi^{(\mathbf{m})}:=E_{\beta_1}^{(m_1)}E_{\beta_2}^{(m_2)}\cdots E_{\beta_N}^{(m_N)}.$$

Then the set 
$$\{E_\bi^{(\mathbf{m})}\mid \mathbf{m}\in\mathbb{N}^N\}$$
forms a $\mathbb{Z}[v,v^{-1}]$-basis of $U_v^+$ (\cite[Proposition 2.3]{Lus1}).

\subsection{Canonical basis of $U_v^+$}
We fix on $\mathbb{N}^N$ the the following orderings: for $\ba=(a_1,\cdots,a_N)$, $\bb=(b_1,\cdots,b_N)\in\mathbb{N}^N$,
\begin{enumerate}
\item lexicographic type ordering $>_L$: $\ba>_L\bb$ if there exists $1\leq i\leq N$ such that $a_1=b_1$, $\cdots$, $a_{i-1}=b_{i-1}$ and $a_i<b_i$;
\item lexicographic type ordering $>_R$: $\ba>_R\bb$ if there exists $1\leq i\leq N$ such that $a_N=b_N$, $\cdots$, $a_{i+1}=b_{i+1}$ and $a_i<b_i$;
\item a partial order $\succ$: $\ba\succ\bb$ if both $\ba>_L\bb$ and $\ba>_R\bb$ hold.
\end{enumerate}

There is another basis of $U_v^+$, whose existence is guaranteed by the following theorem.

\begin{theorem}[\cite{Lus2}]\label{Thm:canonical}
Let $\bi\in R(w_0)$.
\begin{enumerate}
\item For any $\bn\in\mathbb{N}^N$, there exists a unique element $b_\bi^{(\bn)}\in U_v^+$ satisfying the following properties:
\begin{itemize}
\item $b_\bi^{(\bn)}$ is bar-invariant: $\overline{b_\bi^{(\bn)}}=b_\bi^{(\bn)}$;
\item the following upper-triangularity property holds:
$$b_\bi^{(\bn)}-E_\bi^{(\bn)}\in\sum_{\bm\prec\bn}v^{-1}\mathbb{Z}[v^{-1}]E_\bi^{(\bm)}.$$
\end{itemize}
\item The map $\vp_\bi$ sending $\bn$ to $b_\bi^{(\bn)}$ is a bijection between $\mathbb{N}^N$ and a basis $\mathcal{B}$ of $U_v^+$; the basis $\mathcal{B}$ does not depend on the choice of $\bi\in R(w_0)$.
\end{enumerate}
\end{theorem}

This basis $\mathcal{B}$, defined by Lusztig in \cite{Lus1}, is called the canonical basis. This basis was later shown to be the same as the global crystal basis of Kashiwara \cite{Kas}. The map $\vp_\bi:\mathbb{N}^N\ra\mathcal{B}$ is called the Lusztig parametrisation of the canonical basis $\mathcal{B}$ with respect to the reduced decomposition $\bi$, see also \cite{BFZ, BZ2, Cal, CM}.

\subsection{More on the transition matrix}
We fix a reduced decomposition $\bi\in R(w_0)$.

For $\bn\succcurlyeq\bm$, let $\zeta^\bn_\bm\in\mathbb{Z}[v^{-1}]$ and $w^\bn_\bm\in\mathbb{Z}[v,v^{-1}]$ be the base change coefficients given by:

\begin{equation}\label{Eq:canonicalPBW}
b_\bi^{(\bn)}=\sum_{\bm\preccurlyeq\bn}\zeta^\bn_\bm E_\bi^{(\bm)}\ \ \text{and}\ \ \overline{E_\bi^{(\bn)}}=\sum_{\bm\preccurlyeq\bn}w^\bn_\bm E_\bi^{(\bm)}.
\end{equation}
In fact, $\zeta^\bn_\bn=w^\bn_\bn=1$, and for $\bm\prec\bn$, $\zeta^{\bn}_{\bm}\in v^{-1}\mathbb{Z}[v^{-1}]$.

Since the basis elements $b_\bi^{(\bn)}$ are bar-invariant, we have
$$\overline{E_\bi^{(\bn)}}=E_
\bi^{(\bn)}+\sum_{\bm\prec\bn}\left(\zeta^\bn_\bm E_\bi^{(\bm)}-\overline{\zeta^\bn_\bm E_\bi^{(\bm)}}\right).$$
Comparing this with the formula \ref{Eq:canonicalPBW} of the bar-involution, we find: for $\bs\prec\bn$,
\begin{equation}\label{Eq:barcoefficient}
w^\bn_\bs=\zeta^\bn_\bs-\overline{\zeta^\bn_\bs}-\sum_{\bm:\bs\prec\bm\prec\bn}\overline{\zeta^\bn_\bm}w^\bm_\bs.
\end{equation}

According to \cite[Section 9.11]{Lus1}, one can solve the coefficients $\zeta^\bn_\bs$ from the Equation \eqref{Eq:barcoefficient}.

\section{Geometric realisation of quantum groups and reduction of the theorem}\label{sec:geometric}

We consider the reduced decomposition 
$$\bi_+:=(n,n-1,n,n-2,n-1,n,\cdots,1,2,\cdots,n)\in R(w_0),$$ 
for which the induced ordering on the positive roots is
$$\alpha_{n,n},\alpha_{n-1,n},\alpha_{n-1,n-1},\alpha_{n-2,n},\alpha_{n-2,n-1},\alpha_{n-2,n-2},\cdots,\alpha_{1,n},\cdots,\alpha_{1,1}.$$
For a rank tuple $\mathbf{r}$, we define for all $1\leq i\leq j\leq n$
$$m_{i,j}=r_{i,j}-r_{i-1,j}-r_{i,j+1}+r_{i-1,j+1}$$ 
(if $i=0$ or $j=n+1$ we set $r_{i,j}=0$), which we enumerate according to the ordering on positive roots:
$$\mathbf{m}(\mathbf{r})=(m_{n,n},m_{n-1,n},m_{n-1,n-1},\cdots,m_{1,n},\cdots,m_{1,1}),$$
and we associate the corresponding PBW type basis element, resp.~canonical basis element
$$E_+(\mathbf{r})=E_{\bi_+}^{(\mathbf{m}(\mathbf{r}))},\; b_+(\mathbf{r})=b_{\bi_+}^{(\mathbf{m}(\mathbf{r}))}.$$
The monomial
$$M=E_1^{(n)}E_2^{(n-1)}\cdots E_{n-1}^{(2)}E_nE_1E_2^{(2)}\cdots E_{n-1}^{(n-1)}E_n^{(n)}$$
can be expanded into a $\mathbb{Z}[v,v^{-1}]$-linear combination of elements in $\mathcal{B}$.

\begin{proposition} An orbit closure $\overline{\mathcal{O}(\mathbf{r})}\subset U$ is a support of the projective map $\pi:\mathcal{F}_U\rightarrow U$ if and only if the canonical basis element $b_+(\mathbf{r})$ appears with non-zero coefficient in the expansion of $M$.
\end{proposition}

\begin{proof} We use the geometric realisation of $U_v^+(\g)$ developed in \cite{Lus1,Lus3}. Let $\Omega$ be the quiver
$$1\rightarrow 2\rightarrow\cdots\rightarrow n.$$
By \cite[10.17.]{Lus3}, there exists an isomorphism of $\mathbb{Q}(v)$-algebras
$$\lambda_\Omega: U_v^+(\g)\rightarrow\mathcal{K}_\Omega\otimes\mathbb{Q}(v),$$
between $U_v^+(\g)$ and (a scalar extension of) the Grothendieck group of a certain category of perverse sheaves on representation spaces of the quiver $\Omega$ (with an algebra structure given by a certain convolution construction). By \cite[0.3]{Lus3}, in our case of a quiver of Dynkin type $A$, this is just the category of direct sums of shifts of intersection cohomology complexes of orbit closures on these representation spaces; the isomorphism $\lambda_\Omega$ is the inverse of the isomorphism $\mathbf{\Theta}$ of \cite[Proposition 9.8.]{Lus1}.

More precisely, we can translate the geometric setup of Section \ref{subsec:mainresult} to the notation of \cite{Lus3} as follows:

We consider the graded vector space $\mathbf{V}=\bigoplus_{i=1}^nV$. Then, in the notation of \cite[1.2.]{Lus3}, the base space $R$ is $\mathbf{E}_{\mathbf{V}}$, and the group action of $G$ on $R$ coincides with the action of $G_\mathbf{V}$ on $\mathbf{E}_\mathbf{V}$ there. The orbit $\mathcal{O}(\mathbf{r})$ equals $\mathcal{O}_{\mathbf{m}(\mathbf{r})}$ in the notation of \cite[4.15., 4.16.]{Lus1}.

We consider the sequences $$\mathbf{i}=(1,2,\cdots,n,1,2,\cdots,n),\; \mathbf{a}=(n,n-1,\cdots,1,1,2,\dots,n).$$ Working through the definitions of \cite[1.4., 1.5.]{Lus3}, we see that the map $\pi_{\mathbf{i},\mathbf{a}}:\widetilde{\mathcal{F}}_{\mathbf{i},\mathbf{a}}\rightarrow\mathbf{E}_\mathbf{V}$ there is precisely our family $\pi:\mathcal{F}\rightarrow R$ of linear degenerations.

By \cite[Prop. 10.13.]{Lus3}, the isomorphism $\lambda_\Omega$ maps the monomial $M$ to a shift of the complex $\mathrm{R}\pi_*\mathbb{Q}_\mathcal{F}$ on $R$. On the other hand, $\lambda_\Omega$ maps $b_+(\mathbf{r})$ to a shift of the complex ${\rm IC}(\overline{\mathcal{O}(\mathbf{r})})$ by \cite[9.4., Theorem 9.13.]{Lus1}. 

We conclude that ${\rm IC}(\overline{\mathcal{O}(\mathbf{r})})$ appears as a direct summand of $\mathrm{R}\pi_*\mathbb{Q}_\mathcal{F}$ on $R$ if and only if $b_+(\mathbf{r})$ appears with non-zero coefficient in the expansion of $M$ in the canonical basis. Restricting to the open $G$-invariant subset $U\subset R$ yields the desired statement.\end{proof}

\section{Multi-segment duality}\label{sec:msd}

\subsection{Notation}
We fix a reduced decomposition $\bi\in R(w_0)$ and identify $\bm=(m_1,m_2,\cdots,m_N)\in\mathbb{N}^N$ with the PBW basis element $E_{\bi}^{(\bm)}$.

For $1\leq i\leq j\leq n$ we denote $\be_{i,j}$ the coordinate of $\mathbb{N}^N$ corresponding to $E_{\alpha_{i,j}}$.

For $\bm\in\mathbb{N}^N$, we write 
$$\bm=\sum_{1\leq i\leq j\leq n}m_{i,j}\be_{i,j},$$
which is called a multi-segment in \cite{KZ}.

We define a multi-segment $(\ba,\bx)\in\mathbb{N}^n\times\mathbb{N}^{n-1}$, for $\ba=(a_1,\cdots,a_n)$ and $\bx=(x_1,\cdots,x_{n-1})$ in $\mathbb{N}^N$, by:
$$(\ba,\bx):=\sum_{1\leq i\leq n}a_i\be_{i,i}+\sum_{1\leq j\leq n-1}x_j\be_{j,j+1}.$$

We will use the convention $a_0=a_{n+1}=x_0=x_n=0$.

\subsection{Multi-segment duality}

Let $\bi_-\in R(w_0)$ be the reduced decomposition
$$\bi_-:=(1,2,1,3,2,1,\cdots,n,n-1,\cdots,1).$$

The multi-segment duality map is by definition the piecewise linear map $\zeta:\mathbb{N}^N\ra\mathbb{N}^N$ defined by:
$$\zeta:=\vp_{\bi_+}^{-1}\circ\vp_{\bi_-}.$$

For a multi-segment $\bm=\sum_{1\leq i\leq j\leq n}m_{i,j}\be_{i,j}\in\mathbb{N}^N$, we define a rank tuple $(r_{i,j})_{i\leq j}$ by
$$r_{i,j}(\bm):=\sum_{[i,j]\subset [k,\ell]}m_{k,\ell}.$$

As in the previous section, the multi-segment $\bm$ can be recovered from the rank tuple by
$$m_{i,j}=r_{i,j}(\bm)-r_{i-1,j}(\bm)-r_{i,j+1}(\bm)+r_{i-1,j+1}(\bm),$$
where $r_{k,\ell}$ is formally defined as zero if $1\leq k\leq \ell\leq n$ is not fulfilled.

\subsection{Specialization of the Knight-Zelevinsky formula}

By \cite[3., Remark]{BZ}, the multi-segment duality $\zeta$ can be described using the explicit formula \cite[Theorem 1.2.]{KZ} (see also \cite{Z2}). Namely, we have

$$\hat{\br}(\zeta(\bm))=\br(\hat{\zeta}(\bm)),$$
where $\br\mapsto\hat{\br}$ is the involution on rank tuples given by $$\hat{r}_{i,j}=r_{n+1-j,n+1-i},$$ and the map $\hat{\zeta}$ on multi-segments is given by the following formula:
$$r_{i,j}(\hat{\zeta}(\mathbf{m}))=\min_\nu\sum_{(k,l)\in[1,i]\times[j,n]}m_{\nu(k,l)+k-i,\nu(k,l)+l-j},$$
where the sum ranges over all maps $\nu:[1,i]\times[j,n]\rightarrow[i,j]$ such that $\nu(k,l)\leq \nu(k',l')$ whenever $k\leq k'$ and $l\leq l'$.

This formula simplifies considerably in the case where $m_{i,j}=0$ for $j-i\geq 2$. Namely, in this case, a summand contributing to the above sum can be non-zero only if for the indices $k$ and $l$ specifying the summand, we have
$$1\geq (l-j)+(i-k),$$
that is,
$$(k,l)\in\{(i,j),(i-1,j),(i,j+1)\}.$$
Denoting
$$p=\nu(i-1,j),\; q=\nu(i,j),\; r=\nu(i,j+1),$$
the above formula thus reduces to
$$r_{i,j}(\hat{\zeta}(\mathbf{m}))=\min_{i\leq p\leq q\leq r\leq j}(m_{p-1,p}+m_{q,q}+m_{r,r+1}).$$
In particular, this applies to the following situation: given $\bx\in\mathbb{N}^{n-1}$ as above such that $x_i+x_{i+1}\leq n+1$ for all $i$, we define $a_i(\bx)=n+1-x_i-x_{i-1}$ (again using the convention $x_0=0=x_n)$, and finally define the multi-segment
$$\bx'=(\ba(\bx),\bx).$$

We then find the following formula (the special case being easily worked out):

\begin{proposition}\label{prop:minima} For $\bx$ as before, we have
$$r_{i,j}(\hat{\zeta}(\mathbf{x}'))=n+1-\max_{i\leq k\leq l\leq m\leq j}(x_{l-1}+x_l-x_{k-1}-x_m).$$
In particular, we have
$$r_{i,i+1}(\hat{\zeta}(\mathbf{x}'))=n+1-\max(0,x_i-x_{i+1},x_i-x_{i-1}).$$
\end{proposition}

\section{Expansion of the monomial $M$}\label{sec:expansion}

As in Section \ref{sec:geometric}, we consider the following monomial $M\in U_v^+$:
$$M=E_1^{(n)}E_2^{(n-1)}\cdots E_{n-1}^{(2)}E_nE_1E_2^{(2)}\cdots E_{n-1}^{(n-1)}E_n^{(n)}.$$

The goal of this section is to study the expansion of the monomial $M$ into canonical basis elements with respect to the parametrization of the canonical basis induced by the reduced decomposition $\bi_-$. In this section we fix $\bi_-$ to be the reduced decomposition, and will drop this index.

\subsection{Expansion in PBW basis}

For $\beta=\alpha_{i,i+1}$, we will denote $E_{i,i+1}:=E_\beta$ for simplicity: it is given by 
$$E_{i,i+1}=E_{i+1}E_i-v^{-1}E_iE_{i+1}.$$

The following formula can be found in \cite[Section 42.1]{Lus2}: for $a,b,c\in\mathbb{N}$,

\begin{equation}\label{Eq:rank2}
E_i^{(a)}E_{i+1}^{(b)}E_i^{(c)}=\sum_{r=0}^{\min(b,c)}v^{-(b-r)(c-r)}\bino{a+c-r}{a}_vE_i^{(a+c-r)}E_{i,i+1}^{(r)}E_{i+1}^{(b-r)}.
\end{equation}

Let $e_1,e_2,\cdots,e_n,f_1,f_2,\cdots,f_n\in\mathbb{N}$. For a tuple $\bx=(x_1,x_2,\cdots,x_{n-1})\in\mathbb{N}^{n-1}$ satisfying for $i=1,\cdots,n-1$,
$$0\leq x_i\leq \min(e_i,f_{i+1}),$$
we define $E(\bx)$ to be the monomial
$$E_1^{(e_1+f_1-x_1)}E_{1,2}^{(x_1)}E_2^{(e_2+f_2-x_1-x_2)}E_{2,3}^{(x_2)}\cdots E_{n-1}^{(e_{n-1}+f_{n-1}-x_{n-2}-x_{n-1})}E_{n-1,n}^{(x_{n-1})}E_n^{(e_n+f_n-x_{n-1})}.$$

\begin{lemma}\label{Lem:expansionPBW}
We have the following identity:
$$E_1^{(f_1)}\cdots E_n^{(f_n)}E_1^{(e_1)}\cdots E_n^{(e_n)}=\sum_{\bx}v^{-\sum_{i=1}^{n-1}(e_i-x_i)(f_{i+1}-x_i)}\prod_{i=1}^n\bino{e_i+f_i-x_{i-1}-x_i}{f_i-x_{i-1}}_vE(\bx),$$
where the sum ranges over all possible tuples $\bx=(x_1,x_2,\cdots,x_{n-1})\in\mathbb{N}^{n-1}$ satisfying for any $i=1,\cdots,n-1$,
$0\leq x_i\leq \min(e_i,f_{i+1}).$
\end{lemma}

\begin{proof}
We prove the formula by induction on $n$. First we rewrite 
$$E_1^{(f_1)}\cdots E_n^{(f_n)}E_1^{(e_1)}\cdots E_n^{(e_n)}$$ 
into
$$E_1^{(f_1)}\cdots E_{n-1}^{(f_{n-1})}E_1^{(e_1)}\cdots E_{n-2}^{(e_{n-2})}E_n^{(f_n)}E_{n-1}^{(e_{n-1})}E_n^{(e_n)}.$$
After applying equation \eqref{Eq:rank2} to the monomial $E_n^{(f_n)}E_{n-1}^{(e_{n-1})}E_n^{(e_n)}$, the inductive hypothesis can be applied.
\end{proof}

\subsection{Expansion in canonical basis}

%Again we fix the reduced decomposition to be $\bi_-$, and drop this index in this paragraph. 
%On $\mathbb{N}^{n-1}$ we fix the lexicographic total ordering $<$.

We consider the special case of Lemma \ref{Lem:expansionPBW} where $(e_1,e_2,\cdots,e_n)=(1,2,\cdots,n)$ and $(f_1,f_2,\cdots,f_n)=(n,n-1,\cdots,1)$.
Let $\bt_0=(t_1^0,t_2^0,\cdots,t_{n-1}^0)$ be the tuple with $t_k^0=\min(k,n-k)$.

Let $\mathcal{P}$ be the set of tuples $\by\in\mathbb{N}^{n-1}$ such that $\bt_0-\by\in\mathbb{N}^{n-1}$. We define a partial order $>$ on $\mathcal{P}$ by: $\bx\geq\by$ if $\bx-\by\in\mathbb{N}^{n-1}$. Note that $\bx\geq\by$ implies $\bx'\succcurlyeq\by'$, recalling that we associate to $\bx=(x_1,\cdots,x_{n-1})\in\mathcal{P}$ the multi-segment element $\bx'$ in $\mathbb{N}^{N}$ defined by:
$$\bx'=\sum_{k=1}^{n-1}x_k\be_{k,k+1}+\sum_{\ell=1}^n(n+1-x_\ell-x_{\ell-1})\be_{\ell,\ell},$$
where we formally set $x_0=x_n=0$. Then for $\bx=(x_1,x_2,\cdots,x_{n-1})\in\mathbb{N}^{n-1}$, we have
$$E(\bx')=E_1^{(n+1-x_1)}E_{1,2}^{(x_1)}E_2^{(n+1-x_1-x_2)}\cdots E_{n-1}^{(n+1-x_{n-2}-x_{n-1})}E_{n-1,n}^{(x_{n-1})}E_n^{(n+1-x_{n-1})}.$$

%The set $\mathcal{P}$ admits a poset structure $\succ$ by componentwise comparison: two tuples $\br\succcurlyeq\bs$ if and only if $\br-\bs\in\mathbb{N}^{n-1}$. The lexicographic ordering on $\mathbb{N}^{n-1}$ is a linear extension of $\succcurlyeq$: $\br\succcurlyeq\bs$ implies $\br\leq \bs$.

\begin{lemma}\label{Lem:w}
For $\bx,\by\in\mathcal{P}$, the coefficient $w_{\by'}^{\bx'}$ is non-zero only if $\bx\geq\by$. If this is the case, 
$$w^{\bx'}_{\by'}=v^{-\sum_{k=1}^{n-1}\frac{1}{2}(x_k-y_k)(x_k-y_k-1)}(v^{-1}-v)^{\sum_{k=1}^{n-1}(x_k-y_k)}\prod_{k=1}^{n-1}[x_k-y_k]_v!\times$$
$$\times\prod_{k=1}^n\bino{n+1-y_{k-1}-y_k}{x_k-y_k}_v\bino{n+1-x_{k}-y_{k-1}}{x_{k-1}-y_{k-1}}_v,$$
where we formally set $y_0=y_n=x_0=x_n=0$.
\end{lemma}

\begin{proof}
Using the commutation relations $E_{i,i+1}E_i=vE_{i}E_{i,i+1}$ and $E_{i+1}E_{i,i+1}=vE_{i,i+1}E_{i+1}$, it is easy to show that
$$\overline{E_{i,i+1}^{(n)}}=\sum_{k=0}^n v^{-\frac{1}{2}k(k-1)}[k]_v! (v^{-1}-v)^kE_i^{(k)}E_{i,i+1}^{(n-k)}E_{i+1}^{(k)}.$$
Applying of this formula to $E(\bx')$ yields the claimed formula for the coefficients $w_{\by'}^{\bx'}$.
\end{proof}

\begin{corollary}\label{Cor:wgeneral}
Let $\bx\in\mathcal{P}$ and $\bp\in\mathbb{N}^N$. Then $w^{\bx'}_{\bp}\neq 0$ only if there exists $\by\in\mathcal{P}$ such that $\bx\geq\by$ and $\bp=\by'$.
\end{corollary}

In particular, if we write $\bd=\bx-\by=(d_1,\cdots,d_{n-1})$ and set $d_0=d_n=0$, the previous formula reads:
\begin{equation}
w^{\bx'}_{\by'}=v^{-\frac{1}{2}\sum_{k=1}^{n-1}d_k(d_k-1)}(v^{-1}-v)^{\sum_{k=1}^{n-1}d_k}\prod_{k=1}^{n-1}[d_k]_v!\prod_{k=1}^n\bino{a_k+d_k+d_{k-1}}{d_k}_v\bino{a_k+d_{k-1}}{d_{k-1}}_v,
\end{equation}
where $a_k=n+1-x_{k-1}-x_k$.

\begin{lemma}\label{Lem:zeta}
For $\bx,\by\in\mathcal{P}$, the coefficient $\zeta_{\by'}^{\bx'}$ is non-zero only if $\bx\geq\by$.
\end{lemma}

\begin{proof}
By Corollary \ref{Cor:wgeneral}, the formula \eqref{Eq:barcoefficient} reads
$$\zeta^{\bx'}_{\by'}-\overline{\zeta^{\bx'}_{\by'}}=w^{\bx'}_{\by'}+\sum\overline{\zeta^{\bx'}_{\bz'}}w^{\bz'}_{\by'}.
$$
where the sum is over all $\bz\in\mathcal{P}$ such that $\bz\geq\by$ and $\bx'\succ\bz'$.

Assuming $\bx\ngeq\by$, we claim that either $w^{\bz'}_{\by'}=0$ or $\zeta^{\bx'}_{\bz'}=0$. Let us assume that $w^{\bz'}_{\by'}\neq 0$. By Lemma \ref{Lem:w}, $\bz\geq\by$ and hence $\bx\ngeq\bz$. The condition $\bz\geq\by$ allows us to apply induction on the set $\mathcal{P}_{\geq\bx}:=\{\bp\in\mathcal{P}\mid \bp\geq\bx\}$ to show that $\zeta^{\bx'}_{\by'}=0$.

Since $w^{\bx'}_{\by'}=0$, the above claim shows that $\zeta^{\bx'}_{\by'}-\overline{\zeta^{\bx'}_{\by'}}=0$. As $\zeta^{\bx'}_{\by'}\in v^{-1}\mathbb{Z}[v^{-1}]$, it must be zero.
\end{proof}

By Lemma \ref{Lem:expansionPBW}, we can expand $M$ into 
\begin{equation}\label{Eq:ME}
M=\sum_{\by\in\mathcal{P}}\lambda_\by E(\by').
\end{equation}
For $\by=(y_1,\cdots,y_{n-1})$,
\begin{equation}
\lambda_\by=v^{-\sum_{k=1}^{n-1}(k-y_k)(n-k-y_k)}\prod_{k=1}^n\bino{n+1-y_{k-1}-y_k}{k-y_k}_v,
\end{equation}
where we formally set $y_0=y_n=0$.

For $\bx\in\mathbb{N}^{n-1}$, let $b(\bx')$ be the canonical basis element corresponding to the PBW basis element $E(\bx')$.

Since the transition matrix between the canonical basis and the PBW basis is unipotent triangular, we can deduce that 
\begin{equation}\label{Eq:MB}
M=\sum_{\by\in\mathcal{P}}\mu_\by b(\by').
\end{equation}

Combining Equations \eqref{Eq:canonicalPBW}, \eqref{Eq:ME}, \eqref{Eq:MB} and applying a M\"obius type transformation to the partial order $\succ$, we can rewrite $\mu_\by$ as follows in view of Lemma \ref{Lem:zeta}:

\begin{lemma}\label{Lem:mu}
For $\by\in\mathcal{P}$, we have:
\begin{equation}\label{Eq:muxi}
\mu_{\by}=\lambda_{\by}-\sum_{\by<\bz}\lambda_{\bz}\left(\zeta^{\bz'}_{\by'}+\sum_{u\geq 1}(-1)^u\sum_{\by<\bp_1<\cdots<\bp_u<\bz}\zeta_{\by'}^{\bp_1'}\zeta_{\bp_1'}^{\bp_2'}\cdots\zeta_{\bp_u'}^{\bz'}\right).
\end{equation}
\end{lemma}

For $\bx,\by\in\mathcal{P}$, the constants $\zeta_{\by'}^{\bx'}$ can be computed recursively from Equation \eqref{Eq:barcoefficient} once the values of $w_{\by'}^{\bx'}$ are known.

%%%%%%%%%%%%%%%%%%%%%%%%%%%%%%

\subsection{Non-vanishing property of coefficients}\label{subsec:nonvan}

Recall that $\mathcal{M}_n$ is the set of Motzkin paths from $(0,0)$ to $(0,n)$.

\begin{proposition}
For $\by\in\mathcal{M}_n$ we have $\mu_\by\neq 0$.
\end{proposition}

\begin{proof}
Let $\bz\in\mathcal{P}$ such that $\bz>\by$. By Lemma \ref{Lem:mu} and the fact that the coefficients $\zeta_\bm^\bn$ are in $v^{-1}\mathbb{Z}[v^{-1}]$, it suffices to show that the degree of $\lambda_\by$ is greater or equal to the degree of $\lambda_\bz$.

The degree of $\lambda_\by$ is given by:
\begin{eqnarray*}
\deg\lambda_\by&=&\sum_{k=1}^n(k-y_k)(n+1-k-y_{k-1})-\sum_{k=1}^{n-1}(k-y_k)(n-k-y_k)\\
&=&(1-y_1)n+\sum_{k=1}^{n-1}(k+1-y_{k+1})(n-k-y_k)-\sum_{k=1}^{n-1}(k-y_k)(n-k-y_k)\\
&=& (1-y_1)n+\sum_{k=1}^{n-1}(n-k-y_k)(y_k-y_{k+1}+1).
\end{eqnarray*}
Hence the difference of degrees equals

\begin{eqnarray*}
& &\deg\lambda_\by-\deg\lambda_\bz\\
&=& n(z_1-y_1)+\sum_{k=1}^{n-1}((n-k-y_k)(y_k-y_{k+1}+1)-(n-k-z_k)(z_k-z_{k+1}+1))\\
&=& n(z_1-y_1)+\sum_{k=1}^{n-1}((z_k-y_k)(k+1-n-z_{k+1})+(z_{k+1}-y_{k+1})(n-k-y_k)+(z_k^2-y_k^2))\\
&=& n(z_1-y_1)+(z_1^2-y_1^2)+(z_1-y_1)(2-n-z_2)+\sum_{k=2}^{n-1}(z_k-y_k)(z_k+y_k-z_{k+1}-y_{k-1}+2)\\
&=& \sum_{k=1}^{n-1}(z_k-y_k)(z_k-z_{k+1}+y_k-y_{k-1}+2).
\end{eqnarray*}
As $\bz-\by> 0$, we can assume that for $k=1,\cdots,n-1$, we have $z_k=y_k+\alpha_k$ for some $\alpha_k\in\mathbb{N}$. Then the above formula reads:
\begin{eqnarray*}
& &\sum_{k=1}^{n-1}\alpha_k (y_k+\alpha_k-y_{k+1}-\alpha_{k+1}+y_k-y_{k-1}+2)\\
&=&\sum_{k=1}^{n-1} (\alpha_k^2-\alpha_k\alpha_{k+1})+\sum_{k=1}^{n-1}\alpha_k(y_k-y_{k+1}+y_k-y_{k-1}+2)\\
&=&\sum_{k=1}^{n-2}\frac{1}{2}(\alpha_k-\alpha_{k+1})^2+\sum_{k=1}^{n-1}\alpha_k((y_k-y_{k+1}+1)+(y_k-y_{k-1}+1)).
\end{eqnarray*}
Since $\by\in\mathcal{M}_n$, $\deg\lambda_\by-\deg\lambda_\bz$ is non-negative. This finishes the proof.
\end{proof}

\section{Proof and discussion of the main result}\label{sec:proof}

\subsection{Proof of the main result}\label{subsec:proof}

In light of Section \ref{sec:geometric}, our main result Theorem \ref{thm:main} reduces to the following statement:

In the expansion
$$M=\sum_{\mathbf{r}}\gamma_\mathbf{r}b_+(\mathbf{r}),$$
a rank tuple $\mathbf{r}\geq\mathbf{r}^1$ appears with non-zero coefficient $\gamma_\mathbf{r}$ if and only if $\mathbf{r}$ is of the form $\mathbf{r}(\mathbf{x}')$ for a Motzkin path $\mathbf{x}$.

By the previous section, we have
$$M=\sum_{\mathbf{x}\in\mathcal{P}}\gamma_\mathbf{x}b_-(\mathbf{x}'),$$
with $\gamma_\mathbf{x}\not=0$ if $\mathbf{x}\in \mathcal{M}_n$ is a Motzkin path. By the definition of the multi-segment duality, this equation can be rewritten as
$$M=\sum_{\mathbf{x}\in\mathcal{P}}\gamma_\mathbf{x}b_+(\mathbf{\hat{r}}(\hat{\zeta}(\mathbf{x}'))).$$

We thus want to decide for which $\mathbf{x}\in\mathcal{P}$ the inequality $\mathbf{\hat{r}}(\hat{\zeta}(\mathbf{x}'))\geq\mathbf{r}^1$ holds. This condition being invariant under the involution $\hat{\cdot}$, it reduces to the condition $r_{i,i+1}(\hat{\zeta}(\mathbf{x}'))\geq n$. By Proposition \ref{prop:minima}, this is equivalent to $\mathbf{x}$ being a Motzkin path.

We thus find that a rank tuple $\mathbf{r}\geq\mathbf{r}^1$ appears in the above expansion with non-zero coefficient $\gamma_\mathbf{r}$ only if $\mathbf{r}=\mathbf{r}(\mathbf{x})$ for a Motzkin path $\mathbf{x}$. But in this case, we already know that the coefficent $\gamma_{\mathbf{r}}$ is non-zero, and the claim follows.

\subsection{Asymptotics}\label{subsec:asymp}

We would like to decide how large the set of supports is compared to the set of all orbit closures $\overline{\mathcal{O}(\mathbf{r})}$ for $\mathcal{O}(\mathbf{r})\subset U$, at least asymptotically for large $n$.

\begin{proposition} The fraction of the number of supports by the number of all orbit closures tends to zero exponentially fast for $n\rightarrow\infty$.
\end{proposition}

\begin{proof}
The number of supports equals the $n$-th Motzkin number $M_n$, and the number of all orbit closures equals the $n$-th Bell number $B_n$ by \cite[Section 4.2.]{CFFFR}.

% \begin{enumerate}
%\item $M_n$ is defined using the generating function 
%$$M(x):=\frac{1-x-\sqrt{1-2x-3x^2}}{2x^2}=\sum_{n=0}^\infty M_nx^n.$$
%\item $B_n$ is defined by the following Dobi\'nski's formula
%$$B_n=\frac{1}{e}\sum_{k=0}^\infty \frac{k^n}{k!}.$$
%\end{enumerate}
%
%\underline{Proposition}.
%For any $C\in\mathbb{R}$,
%$$\lim_{n\rightarrow \infty}\frac{C^nM_n}{B_n}=0.$$

%\underline{Proof}. 
%We have the following asymptotic for $B_n$ and $M_n$:
%\begin{enumerate}
By \cite[Equation (48b)]{E}, we have $B_n\geq n^{n(1-\zeta_n)}$ where 
$$\zeta_n=O\left(\frac{1}{\log^2 n}\right).$$
%\item In Reference (b), one finds the following asymptotic of $M_n$:
%$$M_n\sim \sqrt{\frac{27}{4\pi n^3}}3^n,$$
We have $M_n\leq O(3^n)$ trivially. It thus suffices to show that for any $C>0$,
$$\lim_{n\rightarrow\infty}\frac{C^n}{n^{n(1-\log^{-2} n)}}=0.$$
If $C=0$ there is nothing to prove. If $C\neq 0$,
as 
$$\frac{C^n}{n^{n(1-\log^{-2} n)}}=\frac{C^n n^{n\log^{-2}n}}{n^n}=\exp(n(\log^{-1}n-\log n+\log C)),$$
when $n\rightarrow\infty$, this term tends to zero.
\end{proof}

%\vskip 10pt
%\underline{References}:
%\begin{enumerate}
%\item[(a)]
%L. Epstein, \textit{A function related to the series for $e^{e^x}$}.  J. Math. Phys. Mass. Inst. Tech. 18, (1939). 153--173. 
%\item[(b)]
%The item contributed by Benoit Cloitre,  A001006, The On-Line Encyclopedia of Integer Sequences (OEIS).
%\end{enumerate}

\subsection{PBW locus}

Inside $U\subset R$ there exists an open subset $U_{\mathrm{PBW}}\subseteq U$, called PBW locus, such that the fibers $\mathrm{Fl}^{f_*}(V)$ over this locus can be naturally identified with Schubert varieties \cite{CFFFR} in some partial flag varieties. The following definition of the PBW locus rephrases Definition 3, Definition 4 and Proposition 2 in \cite{CFFFR}.

\begin{definition}\label{Defn:PBW}
An orbit $\mathcal{O}(\mathbf{r})$ belongs to the PBW locus if the rank tuple $\mathbf{r}=(r_{ij})_{1\leq i\leq j\leq n}$ satisfies the following properties:
\begin{enumerate}
\item for $1\leq k\leq n-1$, we have $r_{k,k+1}\in\{n,n+1\}$;
\item for $1\leq i\leq j\leq n$ with $j-i\geq 2$, we have
$$r_{ij}=n+1-\#\{k\mid i\leq k\leq j-1,\ r_{k,k+1}=n\}.$$
\end{enumerate}
\end{definition}

Notice that these rank tuples are uniquely determined by $r_{k,k+1}$ for $k=1,\cdots,n-1$, hence there are exactly $2^{n-1}$ such orbits.

\begin{proposition}\label{Prop:PBW}
Let $\mathbf{r}$ be a rank tuple such that the orbit $\mathcal{O}(\mathbf{r})$ is contained in the PBW locus. Then $\overline{\mathcal{O}(\mathbf{r})}$ is contained in the support of $\pi:\mathcal{F}_U\ra U$.
\end{proposition}

We start with giving the candidates in Motzkin paths parameterising the orbits in the PBW locus.

\begin{definition}
A Motzkin path $\mathbf{x}=(x_1,\cdots,x_{n-1})\in\mathcal{M}_n$ is said to have a single peak, if there exists $1\leq p\leq n-1$ such that
\begin{enumerate}
\item for any $1\leq s\leq p$, we have $x_{s-1}\leq s$;
\item for any $p\leq t\leq n-1$, we have $x_t\geq x_{t+1}$.
\end{enumerate}
If this is the case, $p$ is called a peak. Let $\mathcal{S}_n$ denote the set of Motzkin paths in $\mathcal{M}_n$ having a single peak.
\end{definition}

\begin{lemma}
The cardinality of $\mathcal{S}_n$ is $2^{n-1}$.
\end{lemma}

\begin{proof}
The proof is executed by induction. We consider some $\mathbf{x}=(x_1,\cdots,x_{n-1})\in\mathcal{S}_n$. If $x_{n-1}=0$ then $(x_1,\cdots,x_{n-2})\in\mathcal{S}_{n-1}$: by induction there are $2^{n-2}$ such paths. If $x_{n-1}=1$ we look at $x_1$: if $x_1=1$ then $(x_2-1,\cdots,x_{n-2}-1)\in\mathcal{S}_{n-2}$, by induction there are $2^{n-3}$ such paths; if $x_1=0$ we continue to look at $x_2$. Repeating this procedure we count the cardinality of $\mathcal{S}_n$:
$$\#\mathcal{S}_n=2^{n-2}+\cdots+2^1+2^0+1=2^{n-1}.$$
\end{proof}

\begin{proof}[Proof of Proposition \ref{Prop:PBW}]
As both $\mathcal{S}_n$ and the number of orbits in the PBW locus have the same cardinality, by the invariance under the involution $\hat{\cdot}$, it suffices to show that for a single peak Motzkin path $\mathbf{x}=(x_1,\cdots,x_{n-1})\in\mathcal{S}_n$, the orbit $\mathcal{O}(\mathbf{r}(\hat{\zeta}(\mathbf{x}')))$ is contained in the PBW locus.

Using Proposition \ref{prop:minima}, we show that for any $1\leq i\leq j\leq n$, the ranks $r_{ij}(\hat{\zeta}(\mathbf{x}'))$ coincide with those given in Definition \ref{Defn:PBW}:
\begin{enumerate}
\item The condition $r_{k,k+1}\in\{n,n+1\}$ is clear by Proposition \ref{prop:minima}. Moreover, $r_{k,k+1}=n$ if and only if $\max(0,x_i-x_{i+1},x_i-x_{i-1})=1$, which is equivalent to either $x_i=x_{i+1}+1$ or $x_i=x_{i-1}+1$.
\item Assume that $1\leq p\leq n-1$ such that $p$ is maximal among the peaks of $\mathbf{x}$: under this assumption $x_{p+1}=x_p-1$.
\begin{itemize}
\item If $1\leq i<j\leq p$, the condition (2) in Definition \ref{Defn:PBW} can be rewritten as: 
\begin{eqnarray*}
r_{ij} &=& n+1-\#\{k\mid i\leq k\leq j-1,\ x_k=x_{k-1}+1\}\\
&=& n+1-(x_{j-1}-x_{i-1}).
\end{eqnarray*}
In this case, the maximum in Proposition \ref{prop:minima} is attained when $l=m=j$ and $k=i$, and the maximum is clearly $x_{j-1}-x_{i-1}$.
\item If $p\leq i<j\leq n-1$, a similar argument shows that $r_{ij}=n+1-(x_i-x_j)$, and the maximum in Proposition \ref{prop:minima} is attained when $l-1=k-1$, $l=i$ and $m=j$, and the maximum is $x_i-x_j$.
\item If $i\leq p\leq j$, the rank $r_{ij}$ in Definition \ref{Defn:PBW} count the number of those $k$ such that either $i\leq k\leq p$ and $x_k=x_{k-1}+1$, or $p\leq k\leq j-1$ and $x_k=x_{k+1}+1$. By considering two cases $x_p=x_{p-1}+1$ and $x_p=x_{p-1}$ we obtain the uniform formula:
$$r_{ij}=n+1-(x_{p-1}-x_{i-1})-(x_p-x_j).$$
We then examine the maximum in Proposition \ref{prop:minima}: $\min(x_{k-1}+x_m)$ is attained when $k=i$ and $m=j$; and $\max(x_{l-1}+x_l)$ is attained when $l=p$. In this case the maximum is $(x_{p-1}-x_{i-1})+(x_p-x_j)$.
\end{itemize}
\end{enumerate}

\end{proof}

\subsection{Remarks}\label{subsec:remarks}

We discuss some limitations of our approach and potential directions for further explorations.

\begin{itemize}
\item  Our approach to the determination of the set of supports is not strong enough to give a general description of the graded vector spaces $V^*(\mathbf{r})$ encoding shifts and multiplicities of intersection cohomology complexes in the decomposition
$$\mathrm{R}\pi_*\mathbb{Q}_{\mathcal{F}_U}\simeq\bigoplus_{\mathbf{r}}{\rm IC}(\overline{\mathcal{O}(\mathbf{r})})\otimes V^*(\mathbf{r}).$$
Namely, the Poincar\'e polynomial of $V^*(\mathbf{r})$ equals (up to shift) the coefficient $\gamma_\mathbf{r}$ above. Since the relevant canonical basis elements $b_+(\mathbf{r})$ are not explicitly known except for small $n$, the coefficients $\gamma_\mathbf{r}$ are not known.

\item This missing information prevents us from applying our main result quantitatively, as a tool to determine the cohomology of the degenerations $\mathrm{Fl}^{f_*}(V)$. Fortunately, this cohomology can be determined using the affine pavings of \cite{CFFFR}.

\item Another main result of \cite{CFFFR} determines the flat locus $U'$ of the family $\pi:\mathcal{F}\rightarrow R$: a fibre $\mathrm{Fl}^{f_*}(V)$ is of dimension $n(n+1)/2$ (but typically reducible) if and only if $\mathbf{r}(f_*)\geq\mathbf{r}^2$ for a certain explicit rank tuple $\mathbf{r}^2$. Our present methods are not strong enough to determine the set of supports of the extended family $\pi:\mathcal{F}_{U'}\rightarrow U'$, since the degree estimates of Section \ref{subsec:nonvan} do not generalize further, as an example for $n=6$ showed.

%\item Our present methods are also not strong enough to give the supports of more general families of linear degenerations of partical flag varieties which will be studied in forthcoming joint work with G. Cerulli Irelli, E. Feigin and G. Fourier.

\item Due to the complicated nature of the Knight-Zelevinsky formula for the multi-segment duality, there seems to be no obvious intrinsic description, in terms of inequalities between the components $r_{i,j}$, for when a rank tuple $\mathbf{r}$ is Motzkin.

\item Since the multi-segment duality admits a geometric interpretation in terms of preprojective varieties for the quiver $\Omega$ by \cite{KZ}, there is, yet unexplored, potential for re-geometrization of our present methods for the proof of the main result.

%\item Inside the locus $U\subset R$, there is a smaller open set, called PBW locus in \cite{CFFFR}, for which all fibres of $\pi$ can be naturally identified with Schubert varieties. Examples suggest that all orbits $\mathcal{O}(\mathbf{r})$ in this locus should correspond to Motzkin rank tuples $\mathbf{r}$.

\item With some effort, it can be proved that the multiplicity space $V^*(\mathbf{r})$, for the orbit $\mathcal{O}(\mathbf{r})$ of maximal codimension, is isomorphic to the cohomology of a point if $n$ is even, and to the cohomology of a projective line if $n$ is odd. We omit the proof here.

\end{itemize}
%%%%%%%%%%%%%%%%%%%%%%%%%%%%%%

%\subsection{Examples}\label{subsec:examples}

%The coefficient $\mu_{\bt_0}$ is $[2]_v!$ if $n$ is odd, otherwise it is $1$. Let $\bt_k=\bt_0-e_k$.

%\begin{proposition}\label{Prop:Depth1}
%For $k=1,\cdots,n-1$, the coefficients $\mu_{\bt_k}$ are given by:
%\begin{enumerate}
%\item when $n$ is even:
%\begin{itemize}
%\item if $k< \frac{n}{2}$ or $k> \frac{n}{2}$, $\mu_{\bt_k}=[2]_v!$;
%\item if $k=\frac{n}{2}$, $\mu_{\bt_k}=[3]_v!$;
%\end{itemize}
%\item when $n$ is odd:
%\begin{itemize}
%\item if $k< \frac{n-1}{2}$ or $k> \frac{n+1}{2}$, $\mu_{\bt_k}=[2]_v!$;
%\item if $k=\frac{n\pm 1}{2}$, $\mu_{\bt_k}=[3]_v!$.
%\end{itemize}
%\end{enumerate}
%\end{proposition}

\section{Small rank examples}\label{sec:smallrank}

We provide explicit results when $n=2,3,4$. Recall that in these cases, the Motzkin numbers are $2,4,9$.

%For $\br=(r_1,\cdots,r_{n-1})\in\mathbb{N}^{n-1}$, we let $\psi_\pm(\br)$ denote the canonical basis element associated to 
%$$\vp_{\bi_\pm}(n+1-r_1,r_1,n+1-r_1-r_2,r_2,\cdots,n+1-r_{n-2}-r_{n-1},r_{n-1},n+1-r_{n-1}).$$

\subsection{The case $n=2$}\label{subsec:rank2}

In this case, let $\bb_1=b((1)')$ and $\bb_2=b((0)')$. Then
$$M=\bb_1+[3]_v!\bb_2,$$
consistent with the calculations in Section \ref{subsec:elex}.

\subsection{The case $n=3$}

In this case, let $\bb_1=b((1,1)')$, $\bb_2=b((1,0)')$, $\bb_3=b((0,1)')$ and $\bb_4=b((0,0)')$. Then

$$M=[2]_v!\bb_1+[3]_v!(\bb_2+\bb_3)+[4]_v!\bb_4.$$

Let $\mathbf{rk}_i$ be the rank tuple associated to $\bb_i$: (in the order $(r_{1,2},r_{1,3},r_{2,3})$)

$$\rk_1=(3,2,3),\ \ \rk_2=(3,3,4),\ \ \rk_3=(4,3,3),\ \ \rk_4=(4,4,4).$$

The orbit corresponding to all these rank tuples lie in the PBW locus, hence they belong to $U$. The number $4$ coincides with the Motzkin number.

\subsection{The case $n=4$}

In this case, let 
$$\bb_1=b((1,2,1)'),\ \ \bb_2=b((1,2,0)'),\ \ \bb_3=b((1,1,1)'),\ \ \bb_4=b((1,1,0)'),$$
$$\bb_5=b((1,0,1)'),\ \ \bb_6=b((1,0,0)'),\ \ \bb_7=b((0,2,1)'),\ \ \bb_8=b((0,2,0)'),$$ 
$$\bb_9=b((0,1,1)'),\ \ \bb_{10}=b((0,1,0)'),\ \ \bb_{11}=b((0,0,1)'),\ \ \bb_{12}=b((0,0,0)').$$

Then
$$M=\bb_1+[2]_v!\bb_2+[3]_v!\bb_3+[3]_v [3]_v[2]_v\bb_4+\bino{4}{2}_v\bb_5+[4]_v[3]_v[3]_v\bb_6+[2]_v!\bb_7+$$
$$+[2]_v!^2\bb_8+[3]_v [3]_v[2]_v\bb_9+[4]_v[3]_v[2]_v[2]_v\bb_{10}+[4]_v[3]_v[3]_v\bb_{11}+[5]_v!\bb_{12}.$$

Let $\mathbf{rk}_i$ be the rank tuple associated to $\bb_i$: (in the order $(r_{1,2},r_{1,3},r_{1,4},r_{2,3},r_{2,4},r_{3,4})$)
$$\mathbf{rk}_{1}=(4,3,2,4,3,4),\ \ \mathbf{rk}_{2}=(4,2,2,3,3,5),\ \ \mathbf{rk}_{3}=(4,4,3,5,4,4),$$
$$\mathbf{rk}_4=(4,3,3,4,4,5),\ \ \mathbf{rk}_5=(4,4,4,5,4,4),\ \ \mathbf{rk}_6=(4,4,4,5,5,5),$$
$$\mathbf{rk}_7=(5,3,2,3,2,4),\ \ \mathbf{rk}_8=(5,3,3,3,3,5),\ \ \mathbf{rk}_9=(5,4,3,4,3,4),$$
$$\mathbf{rk}_{10}=(5,4,4,4,4,5),\ \ \mathbf{rk}_{11}=(5,5,4,5,4,4),\ \ \mathbf{rk}_{12}=(5,5,5,5,5,5).$$

Among them, 
$$\mathbf{rk}_1,\mathbf{rk}_3,\mathbf{rk}_4,\mathbf{rk}_6,\mathbf{rk}_9,\mathbf{rk}_{10},\mathbf{rk}_{11},\mathbf{rk}_{12}$$
are exactly all rank tuples whose orbits belong to the PBW locus. The orbit corresponding to the tuple $\mathbf{rk}_5$ belongs to $U$, and the orbits corresponding to $\mathbf{rk}_2,\mathbf{rk}_7,\mathbf{rk}_{8}$ do not belong to $U$. There are thus $9$ rank tuples with corresponding orbits in $U$, parametrised by Motzkin paths.

\end{document}